\documentclass{amsart}
\usepackage{amssymb,latexsym,ulem,xcolor}
\usepackage{hyperref}

\usepackage{color}

\newtheorem{theorem}{Theorem}[section]
\newtheorem{lemma}[theorem]{Lemma}
\newtheorem{prop}[theorem]{Proposition}
\newtheorem{cor}[theorem]{Corollary}

\theoremstyle{definition}
\newtheorem{definition}[theorem]{Definition}

\theoremstyle{remark}

\numberwithin{equation}{section}

\newcommand{\cA}{{\cal A}}

\newcommand{\cal }{\mathcal }

\newcommand{\forme}[1]{}

\newcommand{\cay}[2]{{\sf Cay}(#1,#2)}

\newcommand{\Z}{{\Bbb Z}}

\newcommand{\aut}[1]{{\sf Aut}(#1)}

\newcommand{\cT}{\mathcal T}

\newcommand{\und}[1]{\underline{#1}}
\newcommand{\laa}{\langle\!\langle}
\newcommand{\raa}{\rangle\!\rangle}

\newcommand{\F}{\Bbb F}

\newcommand{\GL}{{\sf GL}}

\begin{document}
\title[Generalised dihedral CI-groups]{Generalised dihedral CI-groups}

\author{Ted Dobson}
\address{University of Primorska, UP IAM, Muzejski trg 2, SI-6000 Koper, Slovenia and, University of Primorska, UP FAMNIT, Glagolja\c{s}ka 8, SI-6000 Koper, Slovenia}
\email{ted.dobson@upr.si}
\author{Mikhail Muzychuk}
\address{Department of Mathematics, Ben-Gurion University of the Negev, Israel.}
\email{muzychuk@bgu.ac.il}

\author{Pablo Spiga}
\address{Dipartimento di Matematica e Applicazioni, University of Milano-Bicocca, Via Cozzi 55, 20125 Milano, Italy}
\email{pablo.spiga@unimib.it}

\begin{abstract}
In this paper, we find a strong new restriction on the structure of CI-groups. We show that, if $R$ is a generalised dihedral group  and if $R$ is a CI-group, then for every odd prime $p$ the Sylow $p$-subgroup of $R$ has order $p$, or $9$.   Consequently, any CI-group with quotient a generalised dihedral group has the same restriction, that for every odd prime $p$ the Sylow $p$-subgroup of the group has order $p$, or $9$. We also give a counter example to the conjecture that every BCI-group is a CI-group.
\smallskip

\noindent\textbf{Keywords:} CI-group, DCI-group, generalised dihedral, Cayley isomorphism
\end{abstract}
\subjclass[2020]{Primary 05E18, Secondary 05E30}

\maketitle
\section{Introduction}
Let $R$ be a finite group and let $S$ be a subset of $R$.  The \textit{\textbf{Cayley  digraph}} of $R$ with connection set $S$, denoted $\cay{R}{S}$, is the digraph with vertex set $R$ and with $(x, y)$ being an arc if and only if $xy^{-1}\in S$. Now, $\cay{R}{S}$ is said to be a \textit{\textbf{Cayley  isomorphic}} digraph, or \textit{\textbf{DCI-graph}} for short, if whenever $\cay{R}{S}$ is isomorphic to $\cay{R}{T}$, there exists an automorphism $\varphi$ of $R$ with $S^\varphi =T$. Clearly, $\cay{R}{S}\cong \cay{R}{S^\varphi}$ for every $\varphi\in\mathrm{Aut}(R)$ and hence, loosely speaking, for a DCI-graph $\cay{R}{S}$ deciding when a Cayley digraph over $R$ is isomorphic to $\cay{R}{S}$ is theoretically and algorithmically elementary; that is, the solving set for $\cay{R}{S}$ is reduced to simply $\mathrm{Aut}(R)$ (for the definition of a solving set see for example~\cite{Muzychuk2004,Muzychuk1999}). The group $R$ is a \textit{\textbf{DCI-group}} if $\cay{R}{S}$ is a DCI-graph for every subset $S$ of $R$. Moreover, $R$ is a \textit{\textbf{CI-group}} if $\cay{R}{S}$ is a DCI-graph for every inverse-closed subset $S$ of $R$. Thus every DCI-group is a CI-group.

After roughly 50 years of intense research, the classification of DCI- and CI-groups is still open. The current state of the art in this problem is as follows. There exist two rather short lists of candidates for DCI- and CI-groups and it is known that every DCI- and every CI-group must be a member of the corresponding list, see for instance~\cite{Li2002}. Showing that a candidate on the lists of possible DCI- or CI-groups is actually a DCI- or CI-group, though, takes a considerable amount of effort. Just to give an example, the recent paper of Feng and Kov\'{a}cs~\cite{FengK2018} is a tour de force that shows that elementary abelian groups of rank $5$ are DCI-groups.

In this paper we find an unexpected new restriction concerning generalised dihedral groups, and significantly shorten the list of candidates for  CI-groups.
\begin{definition}\label{defeq:2}{\rm
Let $A$ be an abelian group. The \textit{\textbf{generalised dihedral}} group $\mathrm{Dih}(A)$ over $A$ is the group $\langle A, x\mid a^x=a^{-1},\forall a\in A\rangle$. A group is called generalised dihedral if it is isomorphic to some $\mathrm{Dih}(A)$. When $A$ is cyclic, $\mathrm{Dih}(A)$ is called a dihedral group.}
\end{definition}
Our main result is the following.
\begin{theorem}\label{thrm:main}
Let $\mathrm{Dih}(A)$ be a generalised dihedral group over the abelian group $A$. If $\mathrm{Dih}(A)$ is a $\mathrm{CI}$-group,
then, for every odd prime $p$ the Sylow $p$-subgroup of $A$ has order $p$, or $9$. If $\mathrm{Dih}(A)$ is a {\rm DCI}-group, then, in addition, the Sylow $3$-subgroup has order $3$.
\end{theorem}

Generalised dihedral groups are amongst the most abundant members in the list of putative CI-groups. The importance of Theorem~\ref{thrm:main} is the arithmetical condition on the order of such groups, which greatly reduces even further the list of
candidates for CI-groups. We believe that every generalised dihedral group satisfying this numerical condition on its order is a genuine CI-group. (This is in line with the partial result in~\cite{DobsonMS2015}.)  Additionally, this result further reduces to two other groups on the list, whose definitions we now give.

\begin{definition}
Let $A$ be an abelian group such that every Sylow $p$-subgroup of $A$ is elementary abelian.  Let $n\in\{2,4,8\}$ be relatively prime to $\vert A\vert$.  Set $E(A,n) = A\rtimes\Z_n$, where $a^{g} = a^{-1}$, $\forall a\in A$.
\end{definition}

Note that $E(A,2) = \mathrm{Dih}(A)$.  The groups $E(A,4)$ and $E(A,8)$ have centres $Z_1$ and $Z_2$ of order $2$ and $4$, respectively, and $E(A,4)/Z_1 \cong E(A,8)/Z_2\cong\mathrm{Dih}(A)$.  Babai and Frankl \cite[Lemma 3.5]{BabaiF1978} showed that a quotient of a (D)CI-group by a characteristic subgroup is a (D)CI-group, while the first author and Joy Morris \cite[Theorem 8]{DobsonM2015a} showed that a quotient of a (D)CI-group is a (D)CI-group.  Applying either result we have the following.

\begin{cor}
If $E(A,4)$ or $E(A,8)$ is a $\mathrm{CI}$-group, then, for every odd prime $p$ the Sylow $p$-subgroup of $A$ has order $p$ or $9$.
If $E(A,n),n\in\{2,4,8\}$ is a {\rm DCI}-group, then, in addition, $n\neq 8$ and the Sylow $3$-subgroup of $A$ has order $3$.
\end{cor}

Not much is known about which of the groups under consideration in this paper are CI-groups.  Let $p$ be a prime.  Babai \cite[Theorem 4.4]{Babai1977} showed $D_{2p}$ is a CI-group.  The first author \cite[Theorem 22]{Dobson2002} extended this to some special values of square-free integers.  With Joy Morris, the first and third authors \cite{DobsonMS2015} showed that $D_{6p}$ is a CI-group, $p\ge 5$.  Also, Li, Lu, and P\'alfy showed $E(p,4)$ and $E(p,8)$ are CI-groups.

We have one other result of interest, for which we will need an additional definition.

\begin{definition}
Let $G$ be a group, and $S\subseteq G$.  A {\bf Haar graph} of $G$ with connection set $S$ has vertex set $G\times\Z_2$ and edge set $\{\{(g,0),(sg,1)\}:g\in G{\rm\ and\ }s\in S\}$.
\end{definition}

\noindent So a Haar graph is a bipartite analogue of a Cayley graph.  There is a corresponding isomorphism problem for Haar graphs, and if the group $A$ is abelian, it is equivalent to the isomorphism problem for Cayley graphs of generalised dihedral groups ${\rm Dih}(A)$ that are bipartite (for nonabelian groups the problems are not equivalent, as for non-abelian groups Haar graphs need not be transitive), see \cite[Lemma 2.2]{KoikeK2019}.  If isomorphic bipartite Cayley graphs of ${\rm Dih}(A)$ are isomorphic by group automorphisms of $A$, we say $A$ is a {\bf BCI-group}.  It has been conjectured \cite{ArezoomandT2015} that every BCI-group is a CI-group.  We will also give a counter example to this conjecture by showing $\Z_3^k$ is not a BCI-group for every $k\ge 3$, while it is known that $\Z_3^k$ is a CI-group for every $1\le k\le 5$ \cite{Spiga2009}.

\subsection{Some notation}
 Babai~\cite[Lemma 3.1]{Babai1977} has proved a very useful criterion for determining when a finite group is a DCI-group and, more generally, when $\cay{R}{S}$ is a DCI-graph.
\begin{lemma}\label{2}
Let $R$ be a finite group, and let $S$ be a subset of $R$. Then, $\cay{R}{S}$ is a $\mathrm{DCI}$-graph if and only if $\mathrm{Aut}(\cay{R}{S})$ contains a unique conjugacy class of regular subgroups isomorphic to $R$.
\end{lemma}

Let $\Omega$ be a finite set and let $G$ be a permutation group on $\Omega$. An \textit{\textbf{orbital graph}} of $G$  is a digraph with vertex set $\Omega$ and with arc set a $G$-orbit $(\alpha,\beta)^G=\{(\alpha^g,\beta^g)\mid g\in G\}$, where $(\alpha,\beta)\in \Omega\times\Omega$. In particular, each orbital graph has for its arcs one orbit on the ordered pairs of elements of $\Omega$, under the action of $G$. Moreover, we say that the orbital graphs  $(\alpha,\beta)^G$ and $(\beta,\alpha)^G$ are \textit{\textbf{paired}}. When $(\alpha,\beta)^G=(\beta,\alpha)^G$, we say that the orbital graph is \textit{\textbf{self-paired}}.

When $G$ is transitive and $\omega_0\in \Omega$, there exists a natural one-to-one correspondence between the orbits of $G$ on $\Omega\times \Omega$ (a.k.a. orbitals or 2-orbits of $G$) and the orbits of the stabiliser $G_{\omega_0}$ on $\Omega$ (a.k.a. \textit{\textbf{suborbits}} of $G$). Therefore, under this correspondence, we may naturally define paired and self-paired suborbits.

Two subgroups of the symmetric group $\mathsf{Sym}(\Omega)$ are called 2-\textit{\textbf{equivalent}} if they have the same orbitals. A subgroup of $\mathsf{Sym}(\Omega)$ generated by all subgroups 2-equivalent to a given $G\leq\mathsf{Sym}(\Omega)$ is
called the $2$--\textit{\textbf{closure}} of $G$, denoted $G^{(2)}$.

The group $G$ is said to be $2$-\textit{\textbf{closed}} if $G=G^{(2)}$. It is easy to verify that $G^{(2)}$ is a subgroup of $\mathrm{Sym}(\Omega)$ containing $G$ and, in fact, $G^{(2)}$ is the largest (with respect to inclusion) subgroup of $\mathrm{Sym}(\Omega)$ preserving every orbital of $G$.

\section{Construction and basic results}\label{sec:2}
Let $q$ be a power of an  odd prime and let $\F$ be a field of cardinality $q$. We let
\begin{align*}
G&:=\left\{
\begin{pmatrix}
a&x&z\\
0&b&y\\
0&0&c
\end{pmatrix}
\mid x,y,z\in \F, a,b,c\in \{-1,1\}, abc=1
\right\},\\
D&:=\left\{
\begin{pmatrix}
a& ax& ax^2/2\\
0&1&x\\
0&0&a
\end{pmatrix}
\mid x\in\F, a \in\{ -1,1\}\right\},\\
H&:=
\left\{
\begin{pmatrix}
a& 0& x\\
0& a& y\\
0&0&1
\end{pmatrix}
\mid x,y\in \F, a \in\{ -1,1\}\right\},\\
K&:=
\left\{
\begin{pmatrix}
1 &x& y\\
0 &a& 0\\
0&0&a
\end{pmatrix}\mid x,y\in \F, a \in\{ -1,1\}\right\}.
\end{align*}
It is elementary to verify that $G$, $D$, $H$ and $K$ are subgroups of the special linear group $\mathrm{SL}_3(\F)$. Moreover, $D$, $H$ and $K$ are subgroups of $G$, $|G|=4q^3$, $|D|=2q$ and $|H|=|K|=2q^2$. We summarise in Proposition~\ref{260220a} some more facts.

\begin{prop}\label{260220a} The group $D$ is generalised dihedral over the abelian group $(\F,+)$ and, $H$ and $K$ are generalised dihedral over the abelian group  $(\F\oplus\F,+)$. The core of $D$ in $G$ is $1$. Moreover, $$DK=DH=G=HD=KD \hbox{ and }D\cap H=1=D\cap K.$$
\end{prop}
\begin{proof}
The first two assertions follow with easy matrix computations.
Let $$g:=\begin{pmatrix}1&0&0\\0&-1&0\\0&0&-1\end{pmatrix}\in G$$ and observe that
$$g^{-1}\begin{pmatrix}
a& ax& ax^2/2\\
0&1&x\\
0&0&a
\end{pmatrix}g=\begin{pmatrix}a& -ax& -ax^2/2\\0&1&x\\0&0&a\end{pmatrix}.$$As the characteristic of $\mathbb{F}$ is odd, from this it follows that
$$D\cap D^g=\left\langle\begin{pmatrix}-1&0&0\\0&1&0\\0&0&-1\end{pmatrix}\right\rangle.$$
It is now easy to see that $D$ is core-free in $G$.

It is readily seen from the definitions that $D\cap H=1=D\cap K$. Therefore, $|DH|=|D||H|=4q^3$ and $|DK|=|D||K|=4q^3$. As $DH$ and $DK$  are  subsets of $G$ and $|G|=4q^3$, we deduce $DH=G=DK$ and hence also $HD=G=KD$.
\end{proof}

We let $D\backslash G:=\{Dg\mid g\in G\}$  be the set of right cosets of $D$ in $G$. In view of Proposition~\ref{260220a}, $G$ acts faithfully on $D\backslash G$ and $H$ and $K$ act regularly on $D\backslash G$.

\begin{prop}\label{260220aa}
The subgroups $H$ and $K$ are normal in $G$ and, therefore, are in distinct $G$-conjugacy classes.
\end{prop}
\begin{proof}
The normality of $H$ and $K$ in $G$ can be checked by direct computations.
\end{proof}

\subsection{Schur notation}\label{schurnotation}
Since $G=DH$ and $D\cap H=1$, for every $g\in G$, there exists a unique $h\in H$ with $Dg=Dh$.  In this way, we obtain a bijection $\theta:D\backslash G\to H$, where $\theta(Dg)=h\in H$ satisfies $Dg=Dh$.

Using the method of Schur~(see~\cite{Wielandt1964}),  we may identify via $\theta$ the $G$-set $D\backslash G$ with $H$. Moreover,  we may define an action of $G$ on $H$ via the following rule: for every $g\in G$ and for every $h\in H$,
$$
h^g = h'\hbox{  if and only if  } Dhg = Dh'.
$$
A classic observation of Schur yields that the action of $G$ on $D\backslash G$ is permutation isomorphic to the action of $G$ on $H$.
In the rest of the paper, we use both points of view.

In the action of $G$ on $H$, $D$ is a stabiliser of the identity $e\in H$, i.e. $G_e = D$, and $H$ acts on itself via its right regular representation. Since $H$ is normal in $G$, the action of the point stabiliser $G_e$ on $H$ is permutation equivalent to the action of $G_e$ via conjugation on $H$ (Proposition 20.2 \cite{Wielandt1964}). More precisely, $h^g = g^{-1}hg$ for any $g\in G_e$ and $h\in H$.


In what follows, we represent the elements of $H$ and $D$ as pairs $\lbrack a,x\rbrack$ and $\lbrack a,\vec{w}\rbrack$, where $x\in \F$, $\vec{w}\in\F^2$ and $a\in\{\pm 1\}$. In particular, $\lbrack a,x\rbrack$ represents the matrix
$$\begin{pmatrix}
a&ax&ax^2/2\\
0&1&x\\
0&0&a
\end{pmatrix}$$of $D$ and, if $\vec{w}=(x,y)$, then $\lbrack a,\vec{w}\rbrack$ represents the matrix
$$\begin{pmatrix}
a&0&x\\
0&a&y\\
0&0&1
\end{pmatrix}$$of $H$.
Under this identification, the product in $D$ and $H$ greatly simplifies. Indeed, for every $\lbrack a,x\rbrack,\lbrack b,y\rbrack\in D$ and for every $\lbrack a,\vec{v}\rbrack,\lbrack b,\vec{w}\rbrack\in H$, we have
\begin{align}\label{eq2}
\lbrack a,x\rbrack\lbrack b,y\rbrack&=\lbrack ab,bx+y\rbrack,\\
\lbrack a,\vec{v}\rbrack\lbrack b,\vec{w}\rbrack&=\lbrack ab,b\vec{v}+\vec{w}\rbrack.\nonumber
\end{align}
Using this identification, the action of $D$ on $H$ also becomes slightly easier. Indeed, for every $\lbrack a,\vec{v}\rbrack\in H$ (with $\vec{v}=(x,y)$) and for every $\lbrack b,z\rbrack\in D$, we have
\begin{equation}\label{eq1}\lbrack a,(x,y)\rbrack^{\lbrack b,z\rbrack}=\lbrack a,\left((1-a)z^2/2-byz+x,(-1+a)z+by\right)\rbrack.
\end{equation}
This equality can be verified observing that
\[
\begin{pmatrix}
a&0&x\\
0&a&y\\
0&0&1
\end{pmatrix}
\begin{pmatrix}
b&bz&bz^2/2\\
0&1&z\\
0&0&b
\end{pmatrix}
=
\begin{pmatrix}
b&bz&bz^2/2\\
0&1&z\\
0&0&b
\end{pmatrix}
\begin{pmatrix}
a&0&(1-a)z^2/2-byz+x\\
0&a&(-1+a)z+by\\
0&0&1
\end{pmatrix}.
\]
\subsection{One special case}\label{sec:special}Let $A:=\langle e_1,e_2,e_3\rangle$, where $e_1:=(1\,2\,3)$, $e_2:=(4\,5\,6)$,
$e_1:=(7\,8\,9)$, let $x:=(1\,2)(4\,5)(7\,8)$ and let $R:=\langle A,x\rangle$. Then $R$ is a generalised dihedral group over the elementary abelian $3$-group $A$ of order $3^3=27$. Let
$$S:=\{x,\,e_1x,\,e_2x,\,e_3x,\,e_1e_2x,\,e_1^2e_2^2x,\,e_2e_3x,\,e_2^2e_3^2x,\,e_1^2e_2^2e_3^2x\}$$
and define
$$\Gamma:=\cay{R}{S}.$$
It can be verified with a computation with the computer algebra system \texttt{magma} that $\mathrm{Aut}(\Gamma)$ has order $46656=2^6\cdot 3^6$, acts transitively on the arcs of $\Gamma$ and (most importantly) contains two conjugacy classes of regular subgroups isomorphic to $R$ and hence, via Babai's lemma, $R$ is not a CI-group.

This example has another interesting property from the isomorphism problem point of view.  Observe that each element of $S$ is an involution contained in $R\setminus A$.  This implies that $\Gamma$ is a bipartite graph, in which case $\Gamma$ is isomorphic to a Haar graph, also called a bi-coset graph.  
In our example above, as every element of the connection set is an involution, it is a Haar graph of $\Z_3^3$ but as it is not a CI-graph of ${\rm Dih}(\Z_3^3)$, $\Z_3^3$ is not a BCI-group.  This is the first example the authors are aware of where a group is a DCI-group but not a BCI-group, as $\Z_p^3$ is a CI-group \cite{Dobson1995}, and so we have a counter example to the conjecture \cite{ArezoomandT2015} that every BCI-group is a CI-group.   As a quotient of a CI-group is a CI-group, and a quotient of a bipartite Cayley graph of $G$ with bipartition the orbits of $H\le G$ by a quotient of $K\le H$ is still bipartite, $\Z_3^k$ is not a BCI-group for any $k\ge 3$.

Finally, this graph, as well as the  graphs constructed in the next section, have the property that the Sylow $p$-subgroups
of their automorphism groups are not isomorphic to Sylow $p$-subgroups of any $2$-closed group of degree $3^3$ or $p^2$ (in the next section).  For the example above, the Sylow $p$-subgroups of the automorphism groups of Cayley digraphs of $\Z_p^3$ can be obtained from \cite[Theorem 1.1]{DobsonK2009}, and none have order $3^6$ as a Sylow $p$-subgroup of ${\rm AGL}(3,p)$ is not $2$-closed (for $p^2$ in the next section, the Sylow $p$-subgroup has order $p^3$, but Sylow p-subgroups of the automorphism groups of Cayley digraphs of $\Z_p^2$ have order $p^2$ or $p^{p + 1}$ \cite[Theorem 14]{DobsonW2002}).

\section{The permutation group $G$ is $2$-closed}\label{sec:3}
In this section we prove the following.
\begin{prop}\label{26020aaa}The group $G$ in its action on $H$ is $2$-closed.
\end{prop}
We start with some preliminary observations.
\begin{lemma}\label{260220b} The orbits of $G_e$ on $H$ have one of the following forms:
\begin{enumerate}
\item\label{type1} $S_t:=\{\lbrack 1,(t,0) \rbrack\}$, for every $t\in\mathbb{F}$;
\item\label{type2} $C_t\cup C_{-t}$, where $C_t:=\left\{\lbrack 1, (z,t)\rbrack \mid z\in\mathbb{F}\right\}$ and $t\in\F\setminus\{0\}$;
\item\label{type3} $P_t := \left\{\lbrack -1,(t+z^2,2z)\rbrack\mid z\in\mathbb{F}\right\}$ with $t\in\mathbb{F}$.
\end{enumerate}
\end{lemma}
\begin{proof} Let $g:=[a,(x,y)]\in H$. If $a=1$ and $y=0$, then~\eqref{eq1} yields
$$g^{\lbrack b,z\rbrack}=[1,(x,0)]=g$$
and hence the $G_e$-orbit containing $g$ is simply $\{g\}$. Therefore we obtain the orbits in Case~\eqref{type1}.

Suppose then $a=1$ and $y\ne 0$. Now,~\eqref{eq1} yields
\begin{align*}
g^{[1,z]}&=[1,(-yz+x,y)],\\
g^{[-1,z]}&=[1,(yz+x,-y)].
\end{align*}
In particular, $C_y=\{g^{[1,z]}\mid z\in\mathbb{F}\}$ and $C_{-y}=\{g^{[-1,z]}\mid z\in\mathbb{F}\}$ and we obtain the orbits in Case~\eqref{type2}.

Finally suppose $a=-1$.
 Now,~\eqref{eq1} yields
 $$g^{\lbrack b,z\rbrack}=\lbrack 1,(z^2-byz+x,-2z+by)\rbrack.$$
In particular, if we choose $z:=by/2$, then $g$ and $\lbrack -1,(t,0)\rbrack$ (with $t=-y^2/4+x$) are in the same $G_e$-orbit.
Therefore $[-1,(x,y)]^{G_e}=[-1,(t,0)]^{G_e}$ where $t:=-y^2/4+x$.  Using again~\eqref{eq1}, we get
$$\lbrack -1,(t,0)\rbrack^{\lbrack b,-z\rbrack}=\lbrack -1,(t+z^2,2z)\rbrack.$$
In particular, $P_t=\{g^{[b,z]}\mid \lbrack b,z\rbrack\in G_e\}$ and we obtain the orbits in Case~\eqref{type3}.
\end{proof}

We call the $G_e$-orbits in~\eqref{type1} \textit{\textbf{singleton orbits}}, the $G_e$-orbits in~\eqref{type2} \textit{\textbf{coset orbits}} and the $G_e$-orbits in~\eqref{type3} \textit{\textbf{parabolic orbits}}. Clearly, singleton orbits have cardinality $1$, coset orbits have cardinality $2q$ and parabolic orbits have cardinality $q$. Also, it follows from Lemma~\ref{260220b} that there are $q$ singleton orbits, $\frac{q-1}{2}$ coset orbits and $q$ parabolic orbits. Indeed,
$$q\cdot 1+\frac{q-1}{2}\cdot 2q+q\cdot q=2q^2=|H|.$$
 It is also clear from Lemma~\ref{260220b} that  all non-singleton orbits are self-paired and the only self-paired singleton orbit is $S_0$.

Before continuing, we recall~\cite[Definitions~2.5.3 and 2.5.4]{FaradzevKM1994} tailored to our needs.

\begin{definition}{\rm We say that  $h\in H$ \textit{\textbf{separates}} the pair $(h_1,h_2)\in H\times H$,  if $(h,h_1)$ and $(h,h_2)$ belong to distinct $G$-orbitals, that is, $hh_1^{-1}$ and $hh_2^{-1}$ are in distinct $G_e$-orbits.

We also say that a subset $S\subseteq H$ \textit{\textbf{separates}} $G$-orbitals if, for any two distinct elements $h_1, h_2\in H\setminus S$, there exists $s\in S$ separating the pair $(h_1,h_2)$. }
\end{definition}

\begin{prop}\label{290220a} If $q\geq 5$, then $\{e\}\cup P_0$ separates $G$-orbitals.
\end{prop}
\begin{proof} Set $S:=\{e\}\cup P_0$. Let $h_1,h_2\in H\setminus S$ be two distinct elements. If $h_1$ and $h_2$ belong to distinct $G_e$-orbits, then $e\in S$ separates $(h_1,h_2)$. Therefore,  we assume that $h_1$ and $h_2$ belong to the same $G_e$-orbit, say, $O$. Since $h_1\neq h_2$, $O$ is not a singleton orbit and hence $O$ is either a coset or a parabolic orbit.

Assume first that $O$ is a parabolic orbit, that is, $O= P_t$, for some $t\in \mathbb{F}$. By Lemma \ref{260220b}, for each $i\in \{1,2\}$, there exists $x_i\in\mathbb{F}$ with
$h_i = \lbrack -1, (t+x_i^2,2x_i)\rbrack$. As $q=|\mathbb{F}|\ge 5$, it is easy to verify that there exists $x\in \mathbb{F}$ with $x\notin \{x_1,x_2\}$ and with $x-x_1\ne -(x-x_2)$. Now, let $s := \lbrack -1,(x^2,2x)\rbrack\in P_0\subseteq S$. From~\eqref{eq2}, we deduce
$$sh_i^{-1} = \lbrack 1, (t+x_i^2-x^2,2x_i-2x)\rbrack.$$
As $2x_i-2x\ne 0$, from Lemma~\ref{260220b}, we obtain  $sh_i^{-1}\in C_{2(x-x_i)}\cup C_{-2(x-x_i)}$. As $x-x_1\ne -(x-x_2)$, we deduce that $sh_1^{-1}$ and $sh_2^{-1}$ are in distinct $G_e$-orbits and hence $s$ separates $(h_1,h_2)$.

Assume now that $O$ is a coset orbit, that is, $O=C_t\cup C_{-t}$, for some $t\in\F\setminus\{0\}$. In this case, for each $i\in \{1,2\}$, there exist $x_i\in\mathbb{F}$ and $a_i\in \{\pm 1\}$ with
$h_i = \lbrack 1, (x_i,a_it)\rbrack$. Let $x\in \mathbb{F}$ with $$xt(a_2-a_1)\ne x_2 - x_1.$$
(The existence of $x$ is clear when $a_1\ne a_2$ and it follows from the fact that $h_1\ne h_2$ when $a_1=a_2$.)
Set $s := \lbrack -1,(x^2,2x)\rbrack\in P_0\subseteq S$. From~\eqref{eq2}, we have
$$sh_i^{-1}\in \lbrack -1, (x^2 - x_i,2x - a_it)\rbrack.$$ In particular, from Lemma~\ref{260220b}, we have $sh_i^{-1}\in P_{t_i}$, for some $t_i\in \mathbb{F}$. Thus, $(x^2 - x_i,2x - a_it)=(t_i+y^2,2y)$, for some $y\in \mathbb{F}$. From this it follows that  $$t_i =x^2 - x_i - \frac{(2x - a_it)^2}{4}.$$
As $xt(a_2-a_1)\ne x_2 - x_1$, a simple computation yields $t_1\ne t_2$ and hence $sh_1^{-1}$ and $sh_2^{-1}$ are in distinct $G_e$-orbits. Therefore, $s$ separates $(h_1,h_2)$.
\end{proof}

\begin{proof}[Proof of Proposition~$\ref{26020aaa}$]
When $q=3$, the proof follows with a computation with the computer algebra system \texttt{magma}~\cite{Magma}. Therefore, for the rest of the proof we suppose $q\ge 5$. Let $T$ be the $2$-closure of $G$. As $\{e\}\cup P_0$ separates the $G$-orbitals, it follows from~\cite[Theorem~2.5.7]{FaradzevKM1994} that the action of $T_e$ on $P_0$ is faithful, and hence so is the action of $G_e$ on $P_0$. We denote by $G_e^{P_0}$ (respectively, $T_e^{P_0}$) the permutation group induced by $G_e$ (respectively, $T_e$) on $P_0$. In particular, $G_e\cong G_e^{P_0}$ and $T_e\cong T_e^{P_0}$.

 We claim that
\begin{equation}\label{eq3}
(T_e)^{P_0} = (G_e)^{P_0}.
\end{equation}
Observe that from~\eqref{eq3} the proof of Proposition~\ref{26020aaa} immediately follows. Indeed,
$T_e\cong  T_e^{P_0}=G_e^{P_0}\cong G_e$ and hence $T_e=G_e$. As $H$ is a transitive subgroup of $G$, we deduce that $G=G_eH=T_eH=T$ and hence $G$ is $2$-closed. Therefore, to complete the proof, we need only establish~\eqref{eq3}.

From Lemma~\ref{260220b}, $|P_0|=q$. Hence $(G_e)^{P_0}$ is a dihedral group of order $2q$ in its natural action on $q$ points.

For each $t\in \F^*$ let $\Phi_t$ be the subgraph of $\cay{H}{C_t\cup C_{-t}}$ induced by $P_0\rangle$. 
Let $(h_1,h_2)$ be an arc of $\Phi_t$. As $h_1,h_2\in P_0$,
there exist $x_1,x_2\in \mathbb{F}$ with $h_1=[-1,(x_1^2,2x_1)]$ and $h_2=[-1,(x_2^2,2x_2)]$.
Moreover, $h_2h_1^{-1}\in C_t\cup C_{-t}$ and hence, by~\eqref{eq2}, we obtain
$$h_2h_1^{-1}=[1,(x_2^2-x_1^2,2x_2-2x_1)]\in C_t\cup C_{-t},$$
that is, $2x_2-2x_1\in \{-t,t\}$. This shows that the mapping
\begin{align*}
&P_0\to \mathbb{F}^+\\
&(x^2,2x)\mapsto 2x
\end{align*}
is an isomorphism between the graphs $\Phi_t$ and
$\cay{\mathbb{F}^+}{\{-t,t\}}$.
Therefore
$$(G_e)^{P_0}\leq (T_e)^{P_0}\leq
\bigcap_{t\in \mathbb{F}^*}\mathrm{Aut}(\Phi_t)\cong
\bigcap_{t\in \mathbb{F}^*}\mathrm{Aut}(\cay{\mathbb{F}^+}{\{-t,t\}})\cong \mathrm{Dih}(\mathbb{F}^+).$$
Since $(G_e)^{P_0}$ and $\mathrm{Dih}(\mathbb{F}^+)$ are dihedral groups of order $2q$, we conclude that $(G_e)^{P_0} = (T_e)^{P_0} =
\bigcap_{t\in \mathbb{F}^*}\mathrm{Aut}(\Phi_t)$, proving \ref{eq3}.
%
\end{proof}

\section{Generating graph}

Combining Proposition~\ref{26020aaa}, Proposition~\ref{260220aa}, and Lemma \ref{2}, we have proven that ${\rm Dih}(\Z_p^2)$ is not a CI-group with respect to colour Cayley digraphs for odd primes $p$.  In this section we strengthen that result to Cayley graphs.  

\subsection{Schur rings}Let $R$ be a finite group with identity element $e$. We denote the group algebra of $R$ over the field $\mathbb{Q}$ by
$\mathbb{Q}R$. For $Y\subseteq R$, we define $$\underline{Y}:=\sum_{y\in Y}y\in \mathbb{Q}R.$$ Elements
of  $\mathbb{Q}R$ of this form will be called \textit{\textbf{simple quantities}}, see~\cite{Wielandt1964}.
A subalgebra $\mathcal{A}$ of the group algebra $\mathbb{Q}R$ is called a \textit{\textbf{Schur ring}} over $R$ if
the following conditions are satisfied:
\begin{enumerate}
\item\label{nr1} there exists a basis of $\mathcal{A}$ as a $\mathbb{Q}$-vector space consisting of simple quantities $\underline{T}_0,\ldots, \underline{T}_r$;
\item\label{nr2} $T_0 = \{ e \}$, $R=\bigcup_{i=0}^rT_i$ and, for every $i,j\in \{0,\ldots,r\}$ with $i\ne j$, $T_i \cap T_j = ∅\emptyset$;
\item\label{nr3} for each $i\in \{0,\ldots,r\}$, there exists $i'$ such that $T_{i'} = \{ t^{-1} \mid  t \in T_i \}$.
\end{enumerate}
Now, $\underline{T}_0,\ldots,\underline{T}_r$ are called the \textit{\textbf{basic quantities}} of $\mathcal{A}$. A subset $S$ of $R$ is said to be an \textit{\textbf{$\mathcal{A}$-subset}} if $\underline{S}\in\mathcal{A}$, which is equivalent to $S = \bigcup_{j\in J} T_j$,  for some $J\subseteq \{0,\ldots,r\}$.

Given two elements $a:=\sum_{x\in R}a_xx$ and $b:=\sum_{y\in R}b_yy$ in $\mathbb{Q}R$, the \textit{\textbf{Schur-Hadamard}} product $a\circ b$ is defined by
$$a\circ b:=\sum_{z\in R}a_zb_zz.$$
It is an elementary exercise to observe that, if $\mathcal{A}$ is a Schur ring over $R$, then $\mathcal{A}$ is closed by the Schur-Hadamard product.

The following statement is known as the \textit{\textbf{Schur-Wielandt principle}}, see~\cite[Proposition~$22.1$]{Wielandt1964}.
\begin{prop}
Let $\mathcal{A}$ be a Schur ring over $R$, let $q\in \mathbb{Q}$ and let $x:=\sum_{r\in R}a_rr\in\cA$. Then
$$x_q:=\sum_{\substack{r\in R\\ a_r=q}}r\in \cA.$$
\end{prop}
Let $X$ be a permutation group containing a regular subgroup $R$. As in Section~\ref{schurnotation}, we may identify the domain of $X$ with $R$. Let $T_0,\ldots,T_r$ be the orbits of $X_e$ with $T_0=\{e\}$.
A fundamental result of Schur~\cite[Theorem~$24.1$]{Wielandt1964} shows that the $\mathbb{Q}$-vector space spanned by $\underline{T}_0,\underline{T}_1,\ldots,\underline{T}_r$  in $\mathbb{Q}R$ is a Schur
ring over $R$, which is called the \textit{\textbf{transitivity module}} of the permutation group $X$ and is usually
denoted by $V(R,G_e)$. In particular, the $V(R,G_e)$-subsets of the Schur ring $V(R,G_e)$ are a union of $G_e$-orbits.

Let $\mathcal{A}:=\langle\underline{T}_0,\ldots,\underline{T}_r\rangle$ be a Schur ring over $R$ (where $T_0,\ldots,T_r$ are the basic quantities spanning $\mathcal{A}$). The \textit{\textbf{automorphism group}} of $\mathcal{A}$ is defined by
\begin{equation}\label{eq:aut}
\mathrm{Aut}(\mathcal{A}):=\bigcap_{i=0}^r\mathrm{Aut}(\cay{R}{T_i}).
\end{equation}

Given a subset $S$ of $R$, we denote by
$$\laa \underline{S}\raa,$$
the smallest (with respect to inclusion) Schur ring containing $\underline{S}$. Now, $\laa\underline{S}\raa$ is called the \textit{\textbf{Schur ring generated}} by $\underline{S}$.

We conclude this brief introduction to Schur rings recalling~\cite[Theorem~$2.4$]{Muzychuk2003}.

\begin{prop}\label{070320a} Let $S$ be a subset of $R$. Then
$\mathrm{Aut}(\laa \underline{S}\raa)=\mathrm{Aut}(\cay{R}{S})$.
\end{prop}

\subsection{The group $G$ is the automorphism group of a single (di)graph}\label{sec:abla}
It was shown above that the group $G$ is 2-closed, i.e. it is the automorphism of a coloured digraph. In this section we give a Cayley digraph $\cay{H}{T}$ having automorphism group  $G$. To build such a digraph it is sufficient to find a subset $T\subseteq H$ such that  $\laa \underline{T}\raa = V(H,G_e)$ (Proposition~\ref{070320a}). Such a set is constructed in Proposition~\ref{050320a}. Note that $T$ is symmetric for $q\ge 7$, so the digraph $\cay{H}{T}$ is undirected. The cases of $q=3,5$ are exceptional, because in those cases no inverse-closed subset of $H$ has the required property.
\begin{prop}\label{050320a}  Let $q$ be prime, and
\[T:=
\begin{cases}
P_0\cup P_1 \cup P_x\cup C_1\cup C_{-1}&\textrm{where } x\in \mathbb{F}\hbox{ with }x\not\in\{0,\pm 1,\pm 2,\frac{1}{2}\}\hbox{ and }x^6\neq 1,\\
& \hbox{when }q>7,\\
P_0\cup P_1\cup P_3\cup C_1\cup C_{-1}&\textrm{when }q=7,\\
S_1\cup P_0 &\textrm{when }q=5,\\
S_1\cup P_0 &\textrm{when }q=3.
\end{cases}
\]
Then $\laa \und{T}\raa = V(H,G_e)$. In particular, $T$ is not a (D)CI-subset of $H$.
\end{prop}
\begin{proof}
When $q\le 7$, the result follows by computations with the computer algebra system \texttt{magma}. Therefore for the rest of the proof we suppose $q>7$.

According to Proposition~\ref{260220b} the basic sets of $V(H,G_e)$ are of three types: $S_a, C_b\cup C_{-b},P_c$ with $a,b,c\in\F$ and $b\neq 0$. Thus we have three types of basic quantities $\und{S_a}, \und{C_b}+\und{C_{-b}},\und{P_c}$ and $$V(H,G_e)=\langle \und{S_a}, \und{C_b}+\und{C_{-b}},\und{P_c}\,\vline\, a,b,c\in\F,b\neq 0\rangle.$$
Set
\begin{align*}
H_1&:=\{[1,\vec{v}]\mid \vec{v}\in \mathbb{F}^2\},\\
H_2&:=\{[1,(t,0)]\mid t\in \mathbb{F}\}.
\end{align*}
By~\eqref{eq2},  $H_1$ and $H_2$ are subgroups of $H$ with $|H_2|=q$, $|H_1|=q^2$ and, by Lemma~\ref{260220b}, $H_2=\cup_{t\in\mathbb{F}}S_t$.
In Table~\ref{eq:mult} we have reported the multiplication table among the basic quantities of $V(H,G_e)$: this will serve us well.
\begin{table}[ht]\label{eq:mult}
\begin{tabular}{|c||c|c|c|}
\hline
   \  & $\und{S_r}$ & $\und{C_s}$ & $\und{P_t}$\\\hline\hline
$\und{S_a}$ & $\und{S_{a+r}}$ & $\und{C_s}$ & $\und{P_{t-a}}$\\
 \ & \ & \ & \  \\\hline
$\und{C_b}$ & $\und{C_b}$ &
$\begin{cases}
q\und{C_{b+s}} & \textrm{if }b+s\neq 0\\
q\und{H_2} &\textrm{if } b+s=0
\end{cases}
$ & $\und{H\setminus H_1}$\\
 \ & \ & \ & \  \\\hline
$\und{P_c}$ & $\und{P_{c+r}}$ & $\und{H\setminus H_1}$ & $q \und{S_{-c+t}} + \und{H_1\setminus H_2}$\\
 \ & \ & \ & \  \\\hline
\end{tabular}\caption{Multiplication table for the basic quantities of $V(H,G_e)$}
\end{table}

Fix $a,b,c\in \mathbb{F}$ with $b,c\ne 0$ and let $\mathcal{A}$ be the smallest Schur ring of the group algebra $\mathbb{Q}H$ containing $\und{P_a},\und{C_b}+\und{C_{-b}},\und{S_c}$. We claim that
\begin{equation}\label{eq:100}
\mathcal{A}=V(H,G_e).
\end{equation}
Clearly, $\mathcal{A}\le V(H,G_e)$. From Table~\ref{eq:mult}, for every $k\in \{0,\ldots,q-1\}$, we have $\und{S_c}^k=\und{S_{ck}}$ and hence $\und{S_{ck}}\in \cA$. As $c\ne 0$, $\und{S_i}\in\cA$, for each $i\in\{0,\ldots,q-1\}$. Now, as $\und{P_a}\in \cA$, from Table~\ref{eq:mult}, we have $\und{P_{a}}\cdot\und{S_i} =\und{P_{a+i}}\in\cA$ for any $i\in\{0,\ldots,q-1\}$.
The equality $(\und{C_b}+\und{C_{-b}})^2 = 2q\und{H_2} + q\und{C_{2b}}+q\und{C_{-2b}}$ implies $\und{C_{2b}}+\und{C_{-2b}}\in\cA$. Now arguing  inductively we deduce $\und{C_k}+\und{C_{-k}}\in\cA$, for all $k\in\{1,\ldots,q-1\}$. Thus~\eqref{eq:100} follows.

Let $x\in \mathbb{F}$ with $x\not\in\{0,\pm 1,\pm 2,\frac{1}{2}\}$ and $x^6\neq 1$, let $T:=P_0\cup P_1 \cup P_x\cup C_1\cup C_{-1}$ and let $\cT:=\laa\und{T}\raa$. (The existence of $x$ is guaranteed by the fact that $q>7$.) We claim that
\begin{equation}\label{eq:101}
\und{H_2},\,\und{H_1},\,\und{C_2}+\und{C_{-2}},\,\und{S_1}+\und{S_{-1}}+\und{S_x}+\und{S_{-x}} +\und{S_{1-x}}+\und{S_{x-1}}\in \cT.
\end{equation}
Using Table~\ref{eq:mult} for squaring $\und{T}$, we obtain (after rearranging the terms):
\begin{align*}
\und{T}^2 =& 3q\und{S_0}+q\und{S_1}+q\und{S_{-1}}+q\und{S_x}+q\und{S_{-x}} +q\und{S_{1-x}}+q\und{S_{x-1}}\\
& +9\und{H_1\setminus H_2} +12\und{H\setminus H_1} +q\und{C_{2}}+q\und{C_{-2}}+2q\und{H_2}.
\end{align*}
From the assumptions on $x$, the elements $-1,1,-x,x,-(x-1),x-1$ are pairwise distinct.
Therefore
\begin{align*}
\und{T}^2\circ \und{S_b} &=
\begin{cases}
5q\und{S_0}, & b=0,\\
3q\und{S_b}, &\mathrm{if }\, b\in \{\pm 1,\pm x,\pm (x-1)\},\\
2q\und{S_b}, &\mathrm{if }\, b\not\in \{0,\pm 1,\pm x,\pm (x-1)\},
\end{cases}\\
\und{T}^2\circ \und{C_b} &=
\begin{cases}
(q+9)\und{C_b}, &\textrm{if } b\in \{\pm 2\},\\
9\und{C_b}, &\textrm{if } b\not\in \{0,\pm 2\},
\end{cases}\\
\und{T}^2\circ \und{P_b} &= 12 \und{P_b},\quad \textrm{if } b\in\F.
\end{align*}
Since the numbers $6,9,q+9,2q,3q,5q$ are also pairwise distinct (because $q\ne 3$),  an application of the Schur-Wielandt principle yields
\begin{align*}
 (\und{T}^2)_{3q}& = \und{S_1}+\und{S_{-1}}+\und{S_x}+\und{S_{-x}} +\und{S_{1-x}}+\und{S_{x-1}}\in\cT,\\
(\und{T}^2)_{12} &=\und{H\setminus H_1}\in\cT,\\
(\und{T}^2)_{2q} &= \und{H_2} - (\und{S_0} + \und{S_1}+\und{S_{-1}}+\und{S_x}+\und{S_{-x}} +\und{S_{1-x}}+\und{S_{x-1}})\in\cT,\\
(\und{T}^2)_{q+9} &= \und{C_{2}}+\und{C_{-2}}\in\cT.
\end{align*}
From this,~\eqref{eq:101} immediately follows.

We claim  that
\begin{equation}\label{eq:102}\und{S_{1}}+\und{S_{-1}}\in\cT.
\end{equation}
Let $$\cT_{H_2}:=\cT\cap \mathbb{Q}H_2$$
and observe that $\cT_{H_2}$ is a Schur ring over the cyclic group $H_2\cong\mathbb{Z}_q$ of prime order $q$. It is well known that every Schur ring over $\mathbb{Z}_q$ is determined by a subgroup $M\leq\mathrm{Aut}(\Z_q)\cong \mathbb{Z}_q^*$ such that, every  basic set of the corresponding Schur ring is a union of $M$-orbits. This implies that there exists a subgroup $M$ of $\mathrm{Aut}(H_2)$ such that every $\cT_{H_2}$-subset of $\cT_{H_2}$ is a union of $M$-orbits. From~\eqref{eq:101},
the simple quantity $\und{S_1}+\und{S_{-1}}+\und{S_x}+\und{S_{-x}} +\und{S_{1-x}}+\und{S_{x-1}}$ belongs to $\cT_{H_2}$ and hence  $\{\pm 1,\pm x,\pm (1-x)\}$ is a $\cT_{H_2}$-subset of cardinality $6$. It follows that $|M|$ divides six and $M\subseteq \{\pm 1,\pm x,\pm (1-x)\}$. If $|M|\in \{3,6\}$, then  $\{\pm 1,\pm x,\pm (1-x)\}$ is a subgroup of $\Z_q^*$, contrary to the assumption $x^6\neq 1$. Therefore
\begin{equation}\label{eq:1000}
\textrm{either }M=\{1\}\textrm{ or }|M|=\{\pm 1\}.
\end{equation} In both cases, $\{-1,1\}$ is a union of $M$-orbits. Therefore, $\und{S_1} +  \und{S_{-1}}\in \cT_{H_2}$. From this,~\eqref{eq:102} follows immediately.

We are now ready to conclude the proof. Clearly, $\und{T}\in  V(H,G_e)$ and hence $\cT\subseteq V(H,G_e)$.
From~\eqref{eq:101}, $\und{H_1}\in\cT$ and, from~\eqref{eq:102}, $\und{S_1}+\und{S_{-1}}\in \cT$. Therefore $\und{H_1}\circ \und{T} = \und{C_1}+\und{C_{-1}}\in\cT$ and $(\und{T} - \und{H_1})\circ \und{T} = \und{P_0}+\und{P_1}+\und{P_x}\in\cT$. Therefore
$$
\left((\und{P_0}+\und{P_1}+\und{P_x}) (\und{S_1}+\und{S_{-1}})\right)\circ (\und{P_0}+\und{P_1}+\und{P_x}) \in\cT.$$
As
$(\und{P_0}+\und{P_1}+\und{P_x}) (\und{S_1}+\und{S_{-1}})=
\und{P_1}+\und{P_2}+\und{P_{x+1}}+\und{P_{-1}}+\und{P_0}+\und{P_{x-1}}$,
we deduce
$$\left((\und{P_0}+\und{P_1}+\und{P_x}) (\und{S_1}+\und{S_{-1}})\right)\circ (\und{P_0}+\und{P_1}+\und{P_x})=
\und{P_0}+\und{P_1}
$$
and hence $\und{P_0}+\und{P_1}\in\cT$. Therefore, $\und{P_x}=(\und{P_0}+\und{P_1}+\und{P_x})-(\und{P_0}+\und{P_1})\in \cT$.

As $$(\und{P_0}+\und{P_1}) \und{P_x} = q\und{S_x} + q\und{S_{x-1}} + 2(\und{H\setminus H_1}),$$
from the Schur-Wielandt principle, we obtain
$\und{S_x} + \und{S_{x-1}}\in \cT$. Therefore $\und{S_x} + \und{S_{x-1}}\in \cT_{H_2}$ and hence $\{x,x-1\}$ is a $\cT_{H_2}$-subset. Thus $\{x,x-1\}$ is an $M$-orbit. Recall~\eqref{eq:1000}. If $M=\{-1, 1\}$, then $x-1=-1\cdot x=-x$, contrary to the assumption $x\neq 1/2$. Therefore $M=\{1\}$ and $\cT_{H_2}={\mathbb Q}H_2$. Thus $\und{S_i}\in\cT$, for each $i\in\Z_q$.
Thus $\und{S_1},\und{P_x},\und{C_1}+\und{C_{-1}}\in\cT$ and~\eqref{eq:100} implies $V(H,G_e)\subseteq \cT$.
\end{proof}
\section{Proof of Theorem~$\ref{thrm:main}$}

\begin{proof}
The list of candidate CI-groups is on page~323 in~\cite{Li2002}. From here, we see that, if $R$ is in this list and if
$R=\mathrm{Dih}(A)$ is generalised dihedral, then for every odd prime $p$ the Sylow $p$-subgroup of $R$ is
either elementary abelian or cyclic of order $9$.

Assume that the Sylow $p$-subgroup ($p$ is an odd prime) of $A$ is elementary abelian of rank at least $2$. Let $P\leq A$ be a subgroup isomorphic to $\Z_p^2$ and let $x\in R\setminus A$. Then
$\langle P,x\rangle\cong\mathrm{Dih}(\Z_p^2)$. By Proposition~\ref{050320a}, $\mathrm{Dih}(\Z_p^2)$ contains a non-DCI subset. Therefore $\mathrm{Dih}(\Z_p^2)$ is a non-DCI-group. Since subgroups of a (D)CI-group are also (D)CI, we conclude that $R$ is a not a DCI-group as well.  The non-DCI set $T$ constructed in Proposition~\ref{050320a} is symmetric for $p\geq 7$.
Hence $\mathrm{Dih}(\Z_p^2)$ and, therefore, $R$ are non-CI groups when $p\geq 7$. If $p=5$, then the group $\mathrm{Dih}(\Z_p^2)$ contains a non-CI subset, namely: $P_0\cup S_1\cup S_{-1}$ (this was checked by \texttt{magma}\footnote{The automorphism group of the corresponding Cayley graph is $4$ times bigger than $G$ but the subgroups $H$ and $K$ are non-conjugate inside it.}).
Combining these arguments we conclude that if $\mathrm{Dih}(A)$ is a CI-group, then its Sylow $p$-subgroup is cyclic if $p\geq 5$. If $p=3$, then the Sylow $3$-subgroup is either cyclic of order $9$ or elementary abelian. The example in Section~\ref{sec:special} shows that the rank of an elementary abelian group is bounded by $2$.
\end{proof}


We now give the updated list of CI-groups.  It is a combination of the list in~\cite{Li2002}, together with our results here and \cite[Corollary 13]{Dobson2018} (note \cite[Corollary 13]{Dobson2018} contains an error, and should list $Q_8$ on line (1c), not on line (1b)).  We need to define one more group:

\begin{definition}
Let $M$ be a group of order relatively prime to $3$, and $\exp(M)$ be the largest order of any element of $M$.  Set $E(M,3) = M\rtimes_\phi\Z_3$, where $\phi(g) = g^\ell$, and $\ell$ is an integer satisfying $\ell^3\equiv 1\ ({\rm mod}\  \exp(M))$ and $\gcd(\ell(\ell-1),\exp(M)) = 1$.
\end{definition}

\begin{theorem}\label{possible CI-groups}
Let $G$, $M$, and $K$ be {\rm CI}-groups with respect to graphs such that $M$ and $K$ are abelian, all Sylow subgroups of $M$ are elementary abelian, and all Sylow subgroups of $K$ are elementary abelian of order $9$ or cyclic of prime order.
\begin{enumerate}
\item If $G$ does not contain elements of order $8$ or $9$, then $G = H_1\times H_2\times H_3$, where
the orders of $H_1$, $H_2$, and $H_3$ are pairwise relatively prime, and
\begin{enumerate}
\item\label{class 1} $H_1$ is an abelian group, and each Sylow $p$-subgroup of $H_1$ is isomorphic to $\Z_p^k$ for $k < 2p + 3$ or $\Z_4$;
\item\label{class 2} $H_2$ is isomorphic to one of the groups $E(K,2)$, $E(M,3)$, $E(K,4)$, $A_4$, or $1$;
\item\label{class 3} $H_3$ is isomorphic to one of the groups $D_{10}$, $Q_8$, or $1$.
\end{enumerate}
\item\label{class 4} If $G$ contains elements of order $8$, then $G\cong E(K,8)$ or $\Z_8$.
\item\label{class 5} If $G$ contains elements of order $9$, then $G$ is one of the groups $\Z_9\rtimes\Z_2$, $\Z_9\rtimes\Z_4$,
$\Z_2^2\rtimes\Z_9$, or $\Z_2^n\times\Z_9$, with $n\le 5$.
\end{enumerate}
\end{theorem}
\noindent{{\bf Remark.} The rank bound  of an elementary abelian group used in part (1)(a) is due to \cite{Somlai2011}.

Other than positive results already mentioned, the abelian groups known to be CI-groups are $\Z_{2n}$ \cite{Muzychuk1995}, $\Z_{4n}$ \cite{Muzychuk1997}with $n$ an odd square-free integer, $\Z_q\times\Z_p^2$ \cite{KovacsM2009}, $\Z_q\times\Z_p^3$ \cite{MuzychukS2019Preprint}, and $\Z_q\times\Z_p^4$ \cite{KovacsS2019Preprint} with $q$ and $p$ and distinct primes, and $\Z_2^3\times\Z_p$ \cite{DobsonS2013}.  Additional results are given in \cite[Theorem 16]{Dobson2002} and \cite{Dobson2018a} with technical restrictions on the orders of the groups.  A similar result with technical restrictions on $M$ is given in \cite[Theorem 22]{Dobson2002} for some $E(M,3)$.  Also, $E(\Z_p,4)$ and $E(\Z_p,8)$ were shown to be CI-groups in \cite{LiLP2007}, and $Q_8\times\Z_p$ in \cite{Somlai2015}.  Finally, Holt and Royle have determined all CI-groups of order at most $47$ \cite{HoltRpreprint2018}.  Applying Theorem \ref{possible CI-groups} to determine possible CI-groups, and then checking the positive results above to see that all possible CI-groups are known to be CI-groups, we extend the census of CI-groups up to groups of order at most $59$.  We should also add that the isomorphism problem for circulant digraphs has been solved \cite{Muzychuk2004}, and a polynomial time algorithm to determine their automorphism groups has been found \cite{EvdokimovP2003}.  Finally, we remark that the groups $E(M,3)$ and $E(M,8)$ are {\it not} DCI-groups.

\section{Appendix: an alternative approach}
In this section we give an alternative approach to the proof of Theorem~\ref{thrm:main}. We do not give all of the details - just the basic idea. In principle, this section is independent from the previous sections, but for convenience we deduce the main result from our previous work.

For each $g\in \GL_3(\F)$, let $g^\top$ denote the transpose of the matrix $g$ and let $g^\iota:=(g^{-1})^\top$. It is easy to verify that $\iota:\GL_3(\F)\to \GL_3(\F)$ is an automorphism. Let $$s=\begin{pmatrix} 0&0&1\\0&1&0\\1&0&0\end{pmatrix}$$ and let $\alpha$ be the automorphism of $\GL_3(\F)$ defined by
\begin{equation}\label{def:alpha}g^\alpha := s^{-1}g^\iota s=s^{-1} (g^{-1})^{\top} s,
\end{equation}
for every $g\in \GL_3(\F)$.

We now define $\hat{\alpha}\in \mathrm{Sym}(H)$ by
\begin{equation}\label{def:hatalpha}
[a,(x,y)]^{\hat{\alpha}} = [a,(y^2/2-x,ay)],
\end{equation}
for every $[a,(x,y)]\in H$.
\begin{lemma}\label{100620b} Let $\alpha$ and $\hat{\alpha}$ be as in~$\eqref{def:alpha}$ and~$\eqref{def:hatalpha}$. We have
\begin{enumerate}
\item\label{eq:partpar1} $G^\alpha=G$ and $D^\alpha=D$;
\item\label{eq:partpar2} $K=H^\alpha$ and $H=K^\alpha$;
\item\label{eq:partpar3} for every $h\in H$, $(Dh)^\alpha=Dh^{\hat{\alpha}}$;
\item\label{eq:partpar4} for every $x\in \F$ and for every $t\in \F^\ast$, $S_x^{\hat{\alpha}}=S_{-x},C_{t}^{\hat{\alpha}} = C_{t},P_x^{\hat{\alpha}} = P_{-x}$.
\end{enumerate}
\end{lemma}
\begin{proof}
The proof follows from straightforward computations. For every $a\in \{-1,1\}$ and $x\in \F$, we have
\begin{align*}
\begin{pmatrix}
a&ax&ax^2/2\\
0&1&x\\
0&0&a
\end{pmatrix}^\alpha&=
\begin{pmatrix}
0&0&1\\
0&1&0\\
1&0&0
\end{pmatrix}
\left(\begin{pmatrix}
a&ax&ax^2/2\\
0&1&x\\
0&0&a
\end{pmatrix}^{-1}\right)^\top
\begin{pmatrix}
0&0&1\\
0&1&0\\
1&0&0
\end{pmatrix}\\
&=\begin{pmatrix}
0&0&1\\
0&1&0\\
1&0&0
\end{pmatrix}
\begin{pmatrix}
a&-x&a(-x)^2/2\\
0&1&a(-x)\\
0&0&a
\end{pmatrix}^\top
\begin{pmatrix}
0&0&1\\
0&1&0\\
1&0&0
\end{pmatrix}\\
&=\begin{pmatrix}
0&0&1\\
0&1&0\\
1&0&0
\end{pmatrix}
\begin{pmatrix}
a&0&0\\
-x&1&0\\
a(-x)^2/2&a(-x)&a
\end{pmatrix}
\begin{pmatrix}
0&0&1\\
0&1&0\\
1&0&0
\end{pmatrix}\\
&=
\begin{pmatrix}
a&a(-x)&a(-x)^2/2\\
0&1&-x\\
0&0&a
\end{pmatrix}\in D.
\end{align*}
This shows $D^\alpha=D$. The computations for proving $G=G^\alpha$, $K=H^\alpha$ and $H=K^\alpha$ are similar.

Let $h:=[a,(x,y)]\in H$. A direct computation shows that
$$h^\alpha=\begin{pmatrix}a&0&x\\0&a&y\\0&0&1\end{pmatrix}^\alpha=\begin{pmatrix}1&-ay&-ax\\0&a&0\\0&0&a\end{pmatrix}$$
and hence
\begin{align*}
h^\alpha(h^{\hat{\alpha}})^{-1}&=\begin{pmatrix}1&-ay&-ax\\0&a&0\\0&0&a\end{pmatrix}\left(\begin{pmatrix}a&0&y^2/2-x\\0&a&ay\\0&0&1\end{pmatrix}\right)^{-1}\\
&=\begin{pmatrix}1&-ay&-ax\\0&a&0\\0&0&a\end{pmatrix}
\begin{pmatrix}a&0&-ay^2/2+ax\\0&a&-y\\0&0&1\end{pmatrix}\\
&=
\begin{pmatrix}
a&-y&ay^2/2\\
0&1&-ay\\
0&0&a
\end{pmatrix}\in D.
\end{align*}
Therefore $$(Dh)^\alpha=D^\alpha h^\alpha=Dh^\alpha=D h^{\hat{\alpha}}$$
and part~\eqref{eq:partpar3} follows.  Now, part~\eqref{eq:partpar4} follows immediately from Lemma~\ref{260220b} and part~\eqref{eq:partpar3}.
\end{proof}

\begin{lemma}\label{100620a} Let $x\in \mathbb{F}$ with $x\not\in\{0,\pm 1,\pm 2,\frac{1}{2}\}$ and $x^6\neq 1$, and let
\begin{align*}
T&:=P_0\cup P_1 \cup P_x\cup C_1\cup C_{-1},\\
T'&:=P_0\cup P_{-1} \cup P_{-x}\cup C_1\cup C_{-1}.
\end{align*}
Then $\cay{H}{T}$ and $\cay{H}{T'}$ are isomorphic but not Cayley isomorphic. In particular, $H$ is not a $\mathrm{CI}$-group.
\end{lemma}

\begin{proof}We view $G$ as a permutation group on $D\backslash G$, which we may identify with $H$ via the Schur notation.

It follows from Lemma~\ref{100620b}~\eqref{eq:partpar1} and~\eqref{eq:partpar3} that $\hat{\alpha}$ normalizes $G$. Therefore, $\hat{\alpha}$ permutes the orbitals of $G$.
Since $\hat{\alpha}$ fixes $e=[1,(0,0)]$, $\hat{\alpha}$ permutes the suborbits of $G$ and, from Lemma~\ref{100620b}~\eqref{eq:partpar4}, we have
$\cay{H}{T^{\hat{\alpha}}}=\cay{H}{T'}$. Hence $\cay{H}{T}^{\hat{\alpha}}=\cay{H}{T'}$ and $\cay{H}{T}\cong \cay{H}{T'}$.

Assume that there exists $\beta\in\aut{H}$ with $\cay{H}{T}^\beta = \cay{H}{T'}$. Then $\hat{\alpha}\beta^{-1}$ is an automorphism of $\cay{H}{T}$.
It follows from Propositions~\ref{070320a} and \ref{050320a} that $\hat{\alpha}\beta^{-1}\in\mathrm{Aut}(\cay{H}{T})=G$. Therefore $\hat{\alpha}\in G\beta$. Since $G$ and $\beta$ normalize $H$, so does $\alpha$. However, this contradicts Lemma~\ref{100620b}~\eqref{eq:partpar2}.
\end{proof}

On the previous proof, one could prove directly that there exists no automorphism $\beta$ of $H$ with $T^\beta=T'$; however, this requires some detailed computations, in the same spirit as the computations in Section~\ref{sec:abla}.

\bibliography{Generalised_Dihedral_01-08-2020}{}
\bibliographystyle{amsplain}
\end{document}